\documentclass[11pt]{article}

\usepackage[leqno]{amsmath}
\usepackage{amssymb,amsfonts,xfrac,MnSymbol}
\usepackage{amsthm}
\usepackage{hyperref}
\usepackage[capitalise]{cleveref}
\usepackage[textsize=tiny]{todonotes}
\setuptodonotes{fancyline, color=blue!30}
\usepackage{enumitem}
\usepackage{graphicx}
\usepackage{tikz-cd, tikz}
\usepackage{color}
\usepackage{import}
\usepackage{float}
\usepackage{caption}
\usepackage{subcaption}
\numberwithin{equation}{section}

\newtheorem{theorem}{Theorem}[section]
\newtheorem{claim}[theorem]{Claim}
\newtheorem{proposition}[theorem]{Proposition}
\newtheorem{lemma}[theorem]{Lemma}
\newtheorem{corollary}[theorem]{Corollary}
\newtheorem{conj}[theorem]{Conjecture}

\newtheorem*{theorem*}{Theorem}
\newtheorem*{claim*}{Claim}
\newtheorem*{proposition*}{Proposition}
\newtheorem*{lemma*}{Lemma}
\newtheorem*{corollary*}{Corollary}

\theoremstyle{definition}
\newtheorem{definition}[theorem]{Definition}
\newtheorem{observation}[theorem]{Observation}
\newtheorem{remark}[theorem]{Remark}
\newtheorem{example}[theorem]{Example}
\newtheorem{question}[theorem]{Question}

\newtheorem*{definition*}{Definition}
\newtheorem*{observation*}{Observation}
\newtheorem*{remark*}{Remark}
\newtheorem*{example*}{Example}
\newtheorem*{question*}{Question}
\newtheorem*{exercise*}{Exercise}
\newtheorem*{fact*}{Fact}
\newtheorem*{notation*}{Notation}

\newcommand{\bbC}{\mathbb{C}}

\newcommand{\bbQ}{\mathbb{Q}}
\newcommand{\bbR}{\mathbb{R}}

\newcommand{\bbZ}{\mathbb{Z}}

\DeclareMathOperator{\Cay}{Cay}

\DeclareMathOperator{\image}{Im}

\DeclareMathOperator{\spt}{spt}

\DeclareMathOperator{\cd}{cd}
\DeclareMathOperator{\supp}{Supp}
\DeclareMathOperator{\length}{length}
\DeclareMathOperator{\Hull}{Hull}

\DeclareMathOperator{\FP}{FP}
\DeclareMathOperator{\PD}{PD}
\DeclareMathOperator{\FH}{FH}
\DeclareMathOperator{\Lip}{Lip}
\DeclareMathOperator{\Liploc}{Lip_{loc}}

\newcommand{\set}[1]{\left\{ #1 \right\}}

\newcommand*{\claimproofname}{Proof}
\newenvironment{claimproof}[1][\claimproofname]{\begin{proof}[#1]}{\end{proof}}

\usepackage{amssymb,amsfonts,xfrac,MnSymbol}
\crefname{cond}{condition}{conditions}
\creflabelformat{cond}{#2#1\@#3}
\crefname{obs}{observation}{observations}
\creflabelformat{obs}{#2#1\@#3}

\usepackage{commath}
\usepackage{authblk}
\usepackage{blindtext}
\usepackage{enumitem}
\usepackage{soul}
\usepackage[final]{pdfpages}
\usepackage[style=alphabetic]{biblatex}
\usepackage{graphicx} % Required for inserting images
\usepackage{xcolor}
\usepackage{mathrsfs}

\makeatletter
\newtheorem*{rep@theorem}{\rep@title}
\newcommand{\newreptheorem}[2]{%
\newenvironment{rep#1}[1]{%
 \def\rep@title{#2 \ref{##1}}%
 \begin{rep@theorem}}%
 {\end{rep@theorem}}}
\makeatother

\newreptheorem{theorem}{Theorem}

\newcommand{\hammfill}[3]{\operatorname{DFV}_{#1}^{#2, #3}}

\title{Subgroups of word hyperbolic groups in dimension 2 over arbitrary rings}
\date{}
\author{Shaked Bader, Robert Kropholler, Vladimir Vankov}

\begin{document}

\maketitle
\vspace{-10mm}
\centerline{
\textit{\footnotesize{
With an appendix by Shaked Bader
}}
}

\begin{abstract}
In 1996, Gersten proved that finitely presented subgroups of a word hyperbolic group of integral cohomological dimension 2 are hyperbolic. We use isoperimetric functions over arbitrary rings to extend this result to any ring. In particular, we study the discrete isoperimetric function and show that its linearity is equivalent to hyperbolicity, which is also equivalent to it being subquadratic. We further use these ideas to obtain conditions for subgroups of higher rank hyperbolic groups to be again higher rank hyperbolic of the same rank.

The appendix discusses the equivalence between isoperimetric functions and coning inequalities in the simplicial setting and the general setting, leading to combinatorial definitions of higher rank hyperbolicity in the setting of simplicial complexes and allowing us to give elementary definitions of higher rank hyperbolic groups.
\end{abstract}

\section{Introduction}

The most well-behaved hyperbolic groups are locally quasiconvex; these are groups for which every finitely generated subgroup is quasiconvex and, in particular, hyperbolic. Examples of such groups are given by free groups, surface groups and limit groups \cite{WiltonLimit}. Even for groups where subgroups are not quasiconvex, there are hopes to prove that every subgroup is hyperbolic. For instance, in the class of hyperbolic 3-manifold groups, the non-quasiconvex subgroups are known to be surface groups and hence hyperbolic \cite{Canary}. 

However, in general, there are hyperbolic groups with non-hyperbolic finitely generated subgroups. The first examples of this type were given by Rips \cite{Rips}. Subsequently, examples have been constructed satisfying stronger finiteness properties \cite{bradyhyperbolic, IMM, isenrich2021hyperbolic, ClaudioPy}. In particular, \cite{IMM} gives an example of a subgroup of a hyperbolic group with a finite classifying space that is not hyperbolic. 

In low dimensions, these results do not occur. For instance, Gersten \cite{Ger} showed that if $G$ is a hyperbolic group such that $\cd_{\bbZ}(G) = 2$, then a subgroup $H < G$ is hyperbolic if and only if it is of type $\FP_2(\bbZ)$. 
There were attempts to generalise this theorem to general rings. In \cite{MA20} Arora and Martinez-Pedroza managed to weaken the assumption on $G$ and assume only $\cd_{\bbQ}(G)=2$, but they strengthen the assumption on $H$ and assume $H$ is finitely presented. Petrosyan and the third author have improved Gersten's theorem and proved it over the rationals \cite{PetVan}. We prove the theorem over any ring:
\begin{reptheorem}{subgroups over arbitrary rings}
    Let $R$ be an unital ring. Let $G$ be a hyperbolic group such that $\cd_R(G) = 2$. Then a subgroup $H < G$ is hyperbolic if and only if $H$ is of type $\FP_2(R)$.
\end{reptheorem}

As an immediate application, it follows that the non-hyperbolic groups constructed by Rips are not of type $\FP_2(R)$ over any unital ring $R$. This is because they are subgroups of hyperbolic groups that are constructed using small-cancellation theory, thus of cohomological dimension 2 \cite{Rips}.

The proofs of all variations of Gersten's theorem rely on two ideas. Firstly, hyperbolicity is equivalent to some form of a linear isoperimetric inequality. Secondly, the top dimensional isoperimetric function passes to subgroups with strong enough finiteness properties. We begin by proving the latter over an arbitrary normed ring. For definitions and notation, we refer the reader to \cref{sec:normedrings}. 
\begin{reptheorem}{iso bounds for subgroups}
   Let $R$ be a normed ring and $G$ be a group of type $\FP_n(R)$ with $\operatorname{cd}_R(G)=n$. If $H<G$ is a subgroup of type $\FP_n(R)$, then the $R$-homological $(n-1)$-isoperimetric class of $H$ is less than or equal to the $R$-homological $(n-1)$-isoperimetric class of $G$. 
\end{reptheorem}

In \cite{KK20}, it was shown that if $G$ satisfies a linear isoperimetric inequality over a ring $R$ with the discrete norm, then $G$ is hyperbolic. 
We extend this and show that a linear isoperimetric inequality over a ring $R$ with the discrete norm characterises hyperbolicity, see \cref{hyperbolicity implies linear for groups}.

In this paper, we define higher rank hyperbolic groups to be groups with a simplicial classifying space with universal cover that is higher rank hyperbolic in the sense of \cite{KL20}. In the appendix, the first author gives a characterisation of them using the simplicial chain complex, rather than the integral current chain complex used to define higher rank hyperbolic spaces originally. Using that they show
\begin{corollary}
    A non-amenable CAT(0) $\PD_n$ group is rank-$(n-1)$ hyperbolic.
\end{corollary}

We show that higher rank hyperbolic groups satisfy a higher rank analogue of Gersten's theorem (see \cref{sec: higher rank}):
\begin{corollary}
    Let $G$ be a rank-${n}$ hyperbolic group with $\operatorname{cd}_\mathbb{Z}(G)=n+1$. If $H<G$ is a subgroup of type $F_{n+1}$ satisfying $(CI_n)$, then $H$ is rank-$n$ hyperbolic.
\end{corollary}

The paper is organised as follows: in \cref{sec:normedrings} we lay the groundwork for our isoperimetric inequalities, proving foundational results such as invariance under quasi-isometry (cf. \cref{qiinvariant for groups}). In \cref{sec: subgroups}, we study the isoperimetric bounds for subgroups and prove \cref{iso bounds for subgroups}. In \cref{sec: hyp implies lin}, we study the filling functions for hyperbolic groups in the case of a discrete norm on a ring (cf. \cref{hyperbolicity implies linear for groups}). In \cref{sec:subquad implies linear}, we show that for discrete rings a subquadratic isoperimetric function implies a linear isoperimetric function (cf. \cref{subquad implies lin}). In \cref{sec: higher rank}, we study higher rank hyperbolic groups and prove a Gersten-style result in this setting (cf. \cref{higherranksubgroup}). Finally, there is an appendix by the first author that discusses higher rank hyperbolicity and demonstrates the equivalence of isoperimetric functions, as well as coning inequalities, in both the simplicial and integral currents settings.

We thank Dawid Kielak for helpful comments on an earlier draft of the paper.

\section{Homological isoperimetric inequalities}\label{sec:normedrings}

Throughout, $G$ will be a discrete group and $H$ will be a subgroup. We will work over a ring $R$ which we assume to have a unit different from 0. All modules will be left modules. Let $RG$ be the group ring. Recall that $G$ is said to have {\em cohomological dimension $\leq n$ over $R$} if there exists a resolution of the trivial $RG$-module $R$ by projective $RG$-modules of length $n$, that is, there is an exact sequence $$0\to P_n\to\dots\to P_0\to R\to 0$$ where $P_i$ are projective $RG$-modules. We say that $G$ has {\em cohomological dimension $n$ over $R$} if it has cohomological $\leq n$ over $R$ and does not have cohomological dimension $\leq n-1$ over $R$.

\begin{definition} \label{deffpn}
A group is of {\em type $\FP_n(R)$} if there is an exact sequence $$P_n\to \dots\to P_0\to R\to 0$$ where $P_i$ are finitely generated projective $RG$-modules. Equivalently one could assume that $P_i$ are finitely generated free modules. 
\end{definition}
There are two useful consequences of Schanuel's lemma that we will use later. The first is that if $G$ is of cohomological dimension $\leq n$ and we have an exact sequence $$P_{n-1}\to\dots\to P_0\to R\to 0$$ of projective $RG$-modules, then $\ker(P_{n-1}\to P_{n-2})$ is a projective $RG$-module. Secondly, if $G$ is of type $\FP_n(R)$ and $$P_{n-1}\to \dots\to P_0\to R\to 0$$ is an exact sequence of finitely generated projective modules, then $$\ker(P_{n-1}\to P_{n-2})$$ is a finitely generated $RG$-module. 

In order to define isoperimetric functions, we will be using norms on rings.

\begin{definition}
    Let $R$ be a ring. A {\em norm $|\cdot|$} on $R$ is a function $|\cdot|\colon R\to\bbR$ such that:
    \begin{itemize}
        \item $|r|\geq 0$, with equality if and only if $r = 0$, 
        \item $|r+r'|\leq |r|+|r'|$, 
        \item $|rr'|\leq |r||r'|$. 
    \end{itemize}
    We will refer to $R$ as a {\em normed ring}. 
\end{definition}

Given a normed ring $R$ and a free module $F$ with fixed free basis $\Lambda$, we define the {\em $\ell^1$-norm} on $F$ by
\begin{equation}\label{eq:freenorm}
    \begin{split}
        \left|\sum_{x\in\Lambda}a_x x\right|=\sum_{x\in\Lambda}|a_x|
    \end{split}
\end{equation}

Throughout, we will consider $RG$ as a free $R$-module with free basis $G$ and will equip it with the $\ell^1$-norm coming from a norm on $R$. Given a free $RG$-module $M$ with basis $\Lambda$, we endow it with the $\ell^1$-norm from $RG$. We could alternatively consider $M$ as a free $R$-module with basis $\Lambda \times G$ and endow it with the $\ell^1$-norm from $R$. These two point of views are equivalent. 

There are two norms of particular interest. When the ring $R$ is a subring of $\bbC$, we have the restriction of the absolute value, which we denote by $|\cdot|_{\textrm{abs}}$. The second is the discrete norm on $R$, which exists for any ring $R$, given by 
\begin{equation*}
    |r|_d = 
\begin{cases} 
0,\qquad \mbox{if $r = 0$}, \\
1,\qquad \mbox{if $r\neq 0$.}
\end{cases}\end{equation*}

We will use $|\cdot|_{\textrm{abs}}$ for the $\ell^1$-norm on $RG$ coming from $|\cdot|_{\textrm{abs}}$. Similarly, we will use $|\cdot|_d$ for the $\ell^1$-norm on $RG$ coming from $|\cdot|_d$.

We end this section with a useful lemma regarding maps between free modules over a normed ring. 
\begin{lemma}\label{lem:morphismsareboundedwrtnorms}
    Let $R$ be a normed ring and $G$ a group. Let $M = RG^m, N = RG^n$ be finitely generated free $RG$-modules endowed with the $\ell^1$-norm. Suppose that $\phi\colon M\to N$ is an $RG$-homomorphism. Then there exists a constant $C = C(\phi)$ such that $|\phi(x)|\leq C|x|$. 
\end{lemma}
\begin{proof}
    We can consider the map as being given by an $RG$-matrix $(a_{ij})$. Let $x = (x_1, \dots, x_m)$, then $\phi(x) = (\sum_j a_{1j}x_j,\dots,  \sum_j a_{nj}x_j)$. Let $c>0$ be such that $|a_{ij}|\leq c$ for all $i, j$. Computing norms we see that $|\phi(x)|\leq \sum_{i} \sum_j| a_{ij}||x_j| \leq \sum_j nc|x_j| = nc|x|$. Thus we complete the proof by taking $C = nc$.
\end{proof}

\subsection{Filling norms}

Here we give an introduction to filling norms on $RG$-modules, closely following \cite[Section 2]{MA20}.  

\begin{definition}
Let $R$ be a normed ring and let $|\cdot|$ be the corresponding $\ell^1$-norm on $RG$.
Let $\pi \colon F\to M$ be a surjective homomorphism of $RG$-modules, where $F$ is a free $RG$-module with a fixed free basis. The filling norm on $M$ with respect to $\pi$ and $|\cdot|$ is defined as
$$
||m||:= \inf_{x\in F,\pi(x)=m}{|x|}.
$$
\end{definition}

\begin{lemma} [{\cite[Lemma 2.4]{MA20}}] \label{lem:equivalence of filling norms}
Any homomorphism $f\colon M\to N$ between finitely generated $RG$-modules is bounded with respect to filling norms, i.e. there exists $C> 0$ such that $\|f(m)\|^N\leq C\|m\|^M$ for all $m\in M$, where $\|\cdot\|^M, \|\cdot\|^N$ are the filling norms on $M, N$ respectively. 
\end{lemma}

We say that two norms $\|\cdot\|$ and $\|\cdot\|'$, on a module $M$ are linearly or bi-Lipschitz equivalent, and write $\|\cdot\|\sim\|\cdot\|'$ if there exists a constant $C>0$ such that for all $m\in M$ we have
$$
C^{-1}||m||\leqslant ||m||'\leqslant C||m||.
$$

\begin{lemma}[{\cite[Lemma 2.9]{MA20}}]\label{lem:sim norm}
Let M be a finitely generated and projective $RG$-module with filling norm $||\cdot||^M$, and $N$ a finitely generated $RH$-module with filling norm $||\cdot||^N$. If $N$ is an internal direct summand of $M$ (as an $RH$-module), then on $N$ we have
$$
||\cdot||^N\sim||\cdot||^M.
$$
\end{lemma}

\subsection{Isoperimetric functions}

Following Kielak and Kropholler \cite{KK20}, we give the following definition.

\begin{definition}[{\cite[Defintion 2.1]{KK20}}]
    A projective resolution $(C_k,\partial_k)$ of the trivial $RG$-module $R$ is said to be \emph{$n$-admissible} if $C_n,C_{n+1}$ are finitely generated free $RG$ modules equipped with fixed free bases. We think of $\partial_{n+1}\colon C_{n+1}\to C_n$ as a matrix over $RG$ with respect to these bases.
\end{definition}

\begin{remark}
If $G$ is of type $FP_{n+1}(R)$, then there exists an $n$-admissible projective resolution of the trivial $RG$-module $R$ (see \cite{brown}, Proposition VIII.4.3).
\end{remark}

\begin{definition}[Homological $R$-isoperimetric function]
Fix a group $G$, a normed ring $R$. 
Given an $n$-admissible projective resolution $(C_k,\partial_k)$ of the trivial $RG$-module $R$ we have that $\partial_{n+1}\colon C_{n+1}\twoheadrightarrow \image(\partial_{n+1})$ and $C_{n+1}$ is free, so we can consider the corresponding filling norm, $||\cdot||$, that the surjection induces on $\image(\partial_{n+1})$. 
We define the \emph{$n$-isoperimetric function of $(C_k,\partial_k)$ (with respect to $|\cdot|$)} to be the function $f_n\colon\mathbb{R}_{\geq 0}\to\mathbb{R}_{\geq0}\cup\set{\infty}$ defined by
\begin{align*}
f_n(l)&=\sup\{\|b\|\mid b\in \image(\partial_{n+1}),|b|\leq l\}\\
&=\sup\{\inf\set{|c| : \partial_{n+1}(c)=b} \mid b\in \image(\partial_{n+1}),|b|\leq l\}.
\end{align*}
\end{definition}

In the case that the norm $|\cdot|$ takes discrete values the infimum is always obtained. However, it is less clear when the supremum is attained. When $R = \bbZ$ and $|\cdot|$ is the absolute value norm on $\bbZ$, this is the usual homological Dehn function.

In general, it is not clear that one can assume that isoperimetric inequalities take finite values. In \cite{FlemMart}, a more general form of the isoperimetric inequality over permutation $\bbZ G$-modules is studied. They show that this function only takes finite values \cite[Theorem 1.3]{FlemMart}. From this, we obtain that when $R = \bbZ$ with the usual norm, our isoperimetric function takes only finite values. We therefore ask the following question:

\begin{question}
    For which norms $|\cdot|$ does the $n$-isoperimetric function take finite values?
\end{question}

We would like to talk about the \emph{$n$-isoperimetric class of $G$ (with respect to $|\cdot|$)} and for that we need the following definition and lemma that tells us that the class of the isoperimetric function does not depend on the admissible resolution. This is a variation of \cite[Lemma 2.4]{KK20} where it is proved in the case the function is linear.

\begin{definition}
    Given $f,g\colon \mathbb{R}_{\geq 0}\to \mathbb{R}_{\geq 0}\cup{\infty}$ we say that $f\preccurlyeq g$ if there exist constants $C,K> 0$ such that $f(x)\leq Kg(Kx+C)+Kx+C$. We say $f\approx g$ if $f\preccurlyeq g$ and $g\preccurlyeq f$.

    Here we use the convention that for any $K>0$ we have $K+\infty = \infty = K\infty$.

    We note that all polynomials of a given degree are in the same class.
\end{definition}

We will show in the following lemma that the class of the isoperimetric function does not depend on the projective resolution. 

\begin{lemma} \label{iso class is independent}
    If $\textbf{C}=(C_k,\partial_k)$ and  $\textbf{C'}=(C'_k,\partial'_k)$ are two $n$-admissible projective resolutions of the trivial $RG$-module $R$ and $f:=f_n,f':=f_n'$ are their corresponding $n$-isoperimetric functions with respect to $|\cdot|$, then $f\approx f'$.
\end{lemma}

\begin{proof}
       Throughout we will abbreviate $\partial_{n+1}$ to $\partial$ and $\partial'_{n+1}$ to $\partial'$ and we will denote $\image(\partial)=B$ and $\image(\partial')=B'$.
       
       Since $\textbf{C}$,$\textbf{C'}$ are both projective resolutions of $R$, there exist chain maps $\xi\colon \textbf{C}\to \textbf{C'},\zeta\colon \textbf{C'}\to \textbf{C}$ such that $\xi\circ \zeta$ is homotopic to the identity. Let $h$ be such a homotopy. 
       In particular, we have the following diagram
       \begin{equation*}
           \begin{tikzcd}
	{C_{n+1}} && B & B && {C_n} \\
	&&& {} \\
	{C'_{n+1}} && {B'} & {B'} && {C'_n}
	\arrow["{\xi_{n+1}}"', from=1-1, to=3-1]
	\arrow["{\xi_n|_B}"', from=1-3, to=3-3]
	\arrow["\partial", from=1-1, to=1-3]
	\arrow["{\partial'}", from=3-1, to=3-3]
	\arrow["h|_B"', from=1-3, to=3-1]
	\arrow[hook, from=1-4, to=1-6]
	\arrow[hook, from=3-4, to=3-6]
	\arrow["{\zeta_n|_B}", from=3-4, to=1-4]
	\arrow[Rightarrow, no head, from=1-3, to=1-4]
	\arrow[Rightarrow, no head, from=3-3, to=3-4]
        \end{tikzcd}
       \end{equation*}

       Let $\epsilon>0$ and let $b'\in B'$. We have that $\zeta_n(b')\in B$, so there exists $c\in C_{n+1}$ that satisfies $\partial(c)=\zeta_n(b')$ and $|c|\leq f(|\zeta_n(b')|)+\epsilon$. Clearly 
       \[\partial' \xi_{n+1} (c)=\xi_n(\partial(c))=\xi_n\circ \zeta_n (b')=b'-\partial'(h(b')),\] hence $\partial'\big(\xi_{n+1}(c)+h(b')\big)=b'$ and \[|\xi_{n+1}(c)+h(b')|\leq |\xi_{n+1}(c)|+|h(b')|.\]

       We have that $\xi_{n+1},\zeta_n$ and $h$ can be represented by matrices $X,Z,H$ respectively, and we get that 
       \begin{equation*}
       \begin{split}
           |\xi_{n+1}(c)|+|h(b')| &\leq |X|\cdot|c|+|H|\cdot|b'|\\&
           \leq |X|f(|\zeta_n(b')|)+\epsilon+|H|\cdot|b'|\\&
           \leq |X|f(|Z|\cdot|b'|))+|H|\cdot|b'|+\epsilon .
       \end{split}
       \end{equation*}
       As $\epsilon>0$ was arbitrary, we get that $||b'||\leq f(|Z|\cdot|b'|))+|H|\cdot|b'|$, hence $f'(l)=\sup_{b'\in B,|b'|\leq l}||b'||\leq f(|Z|\cdot l))+|H|\cdot l$, so by definition $f'\preccurlyeq f$. Of course, the argument is symmetric, so we get that $f\approx f'$.   
\end{proof}

Following the lemma, we can define the $n$-isoperimetric class of $G$.

\begin{definition}[Isoperimetric inequality]
    Let $R$ be a normed ring and $G$ be a group of type $FP_{n+1}(R)$. The \emph{$n$-isoperimetric class of $G$ over $R$}, denoted $f_n^G$ is the equivalence class of $f_n$ with respect to $\approx$ corresponding to some $n$-admissible projective resolution of the trivial $RG$-module $R$.
    
    There are two special cases, if $R$ is a subring of $\bbC$ and $|\cdot| = |\cdot|_{\textrm{abs}}$ in which case the notation $\operatorname{FV}_n^{R,G}$ is standard for the isoperimetric class of $G$ over $R$, here $\operatorname{FV}$ stands for ``filling volume''. 
    In the case that $R$ has the discrete norm we will denote the corresponding isoperimetric class by $\hammfill{n}{R}{G}$, where DFV stands for ``discrete filling volume''.
    Superscripts and subscripts will be omitted when the context is clear.
    
    When $R$ is fixed we will also call $f_n^G$ the \emph{$n$-isoperimetric class of $G$}.
    
    We will say that $G$ admits a linear, quadratic, or sub-Euclidean \emph{isoperimetric inequality} over $R$ with respect to $|\cdot|$ if $f_n^G$ is linear, $f_n^G$ quadratic, or $f_n^G\preccurlyeq x^{\frac{n+1}{n}}$ respectively.

    Furthermore, we say that $G$ admits a \emph{subquadratic isoperimetric inequality} over $R$ with respect to $|\cdot|$ if $f_n^G\prec x^2$ and $f_n^G\not\approx x^2$.
\end{definition}

We finish this section by showing that if $f_i$ is finite valued for $i
\leq n$, then $f_n$ is a quasi-isometry invariant for groups of type $\FH_n(R)$, see \cref{FH}. We require lemmas analogous to Lemmas 12 and 13 from \cite{AlonsoWangPride}. The proofs are essentially identical using chain homotopies rather than homotopic maps. 

    \begin{definition} \label{def isop for spaces}
		Let $X$ be a simplicial complex such that $H_{i}(X; R) = 0$ for $i\leq n$. We define the {\em $n$-isoperimetric function of $X$ over $R$} to be the $n$-isoperimetric function associated to the reduced $R$-chain complex of $X$. We will denote this function by $f_n^{R, X}$.
	\end{definition}

	\begin{lemma} \label{lem:replacewithcontinuous}
	Let $X, Y$ be two simplicial complexes such that for all $i\leq n$, $H_i(X; R) = H_i(Y; R) = 0$ and $f_i^{R, X}(l), f_i^{R, Y}(l)< \infty$ for all $l$. If $X^{(0)}$ and $Y^{(0)}$ are quasi-isometric and $g\colon X^{(0)}\to Y^{(0)}$ is a $(K, K)$-quasi-isometry, then for $i\leq n+1$ there exists a map $g_i\colon C_i(X;R)\to C_i(Y;R)$ with $g_{i-1}\partial = \partial g_i$ extending $g := g_0\colon C_0(X;R)\to C_0(Y;R)$ and constants $D_i$ such that $|g_i(x)|\leq D_i|x|$.
	\end{lemma}
	\begin{proof}
		We prove this by induction. We start in degree 1.  For each 1-cell $\sigma$ with endpoint $v, w$ pick a path $\gamma_{vw}$ in $Y$ from $g(v)$ to $g(w)$. Since $g$ is a $(K, K)$-quasi-isometry, we can arrange that the length of $\gamma_{vw}$ is $\leq 2K$. Consider $\gamma_{vw}$ as a 1-chain in $Y$ and extend $g$ to $C_1(X; R)$ by sending $\sigma$ to $\gamma_{vw}$. As a 1-chain we can see that $\gamma_{vw}$ satisfies $|\gamma_{vw}|\leq 2K|1_R| =: D_1$.

  		We now assume that the map has been defined for $i<k\leq n+1$. Let $\sigma$ be a $k$-simplex in $X$. We have that $\partial g_{k-1}(\partial\sigma)=g_{k-2}(\partial\partial\sigma)=0$, so $g_{k-1}(\partial\sigma)$ is a cycle. Since $H_k(Y; R)$ vanishes, there is a chain $c\in C_k(Y; R)$ such that $\partial(c) = g_{k-1}(\partial(\sigma))$ with $|c|\leq f_n^{R,X}(|\partial(\sigma)|)+1$. We define $g_k\colon C_k(X; R)\to C_k(Y; R)$ to be the linear map such that $g_k(\sigma) = c$.
		
		By the definition of $g_k$ we see that $\partial(g_k(\sigma)) = \partial(c) = g_{k-1}(\partial(\sigma))$. Also $\partial(\sigma) = \sum_{i=0}^k \sigma_i$, where $\sigma_i$ is a $(k-1)$-simplex. By induction we have that $|g_{k-1}(\sigma_i)|\leq D_{k-1}$ and hence $|g_{k-1}(\partial(\sigma))| \leq (k+1)D_{k-1}$. By construction \[|g_{k}(\sigma)|=|c|\leq f_k^{R,X}(|\partial \sigma|)+1= f_k^{R,X}((k+1)D_{k-1})+1 =: D_k.\qedhere\] 
	\end{proof}

    \begin{lemma}\label{lem:boundedchainhomotopy}
        Let $X, Y$ be as in \Cref{lem:replacewithcontinuous}.
        Let $g\colon X^{(0)}\to Y^{(0)}$ and $h\colon Y^{(0)}\to X^{(0)}$ be $(K, K)$-quasi-isometries such that $h\circ g$ is $K$-close to $id_X$. Let $g_*\colon C_*(X; R)\to C_*(Y; R)$ and $h_*\colon C_*(Y; R)\to C_*(X; R)$ be the partial chain maps obtained from \Cref{lem:replacewithcontinuous} and $D_i,D'_i$ the constants satisfying $|g_i(x)|\leq D_i|x|$ and $|h_i(x)|\leq D'_i|x|$. 
        
        There exist maps $s_i\colon C_i(X; R)\to C_{i+1}(X; R)$ for $i\leq n$ such that $s_{i-1}\partial + \partial s_{i} = id - h_ig_i$ and constants $E_i$ such that $|s_i(\sigma)|\leq E_i$ for all $\sigma\in C_i(X; R)$.
    \end{lemma}
    \begin{proof}
        We construct the maps $s_i$ inductively. Since $h\circ g$ is $K$-close to $id_X$, given a vertex $x\in X$ we define $s_1(x)$ to be the chain corresponding to a shortest path from $x$ to $hg(x)$. Thus $|s_1(x)|\leq K|1_R| = E_0$. 

        Suppose that we have defined $s_{i-1}\colon C_{i-1}(X; R)\to C_i(X;R)$. 
        Given a cell $\sigma\in C_i(X; R)$ we note that $\sigma - h_i(g_i(\sigma)) - s_{i-1}\partial(\sigma)$ is a cycle, as by the induction hypothesis on $s_{i-1}$ we have
        \[\partial s_{i-1}(\partial \sigma)=\partial\sigma-\partial h_i(g_i(\sigma))=\partial(\sigma-h_i(g_i(\sigma))).\]
        We define $s_i(\sigma)$ to be a filling of $\sigma - h_i(g_i(\sigma)) - s_{i-1}\partial(\sigma)$ such that 
        \[|s_i(\sigma)|\leq f_n^{R, X}(|\sigma - h_i(g_i(\sigma)) - s_{i-1}\partial(\sigma)|)+1\leq f_n^{R, X}(1+D_iD'_i+(i+1)E_{i-1})+1\]
        Thus the proof is complete by setting $E_i = f_i^{X}(1 - D_iD'_i + (i+1)(E_{i-1}))+1$.
    \end{proof}

\begin{theorem}
        If $X, Y$ are simplicial complexes such that for all $i\leq n$, $H_i(X; R) = H_i(Y; R) = 0$ and $f_i^{R, X}(l), f_i^{R, Y}(l)< \infty$ for all $l$, then $f_n^{R, X}\approx f_n^{R, Y}$.
    \end{theorem}
    \begin{proof}
    Let $g_*\colon C_*(X;R)\to C_*(Y;R)$, $h_*\colon C_*(Y;R)\to C_*(X;R)$ be the partial chain maps obtained from \Cref{lem:replacewithcontinuous} and $D_i,D'_i$ the constants satisfying $|g_i(x)|\leq D_i|x|$ and $|h_i(x)|\leq D'_i|x|$. 
    
    Let $s_i\colon C_i(X;R)\to C_{i+1}(X;R)$ and $E_i$ be the maps and constants obtained in \cref{lem:boundedchainhomotopy}.
        Let $z\in C_n(X;R)$ be an $n$-cycle. Then $g_n(z)$ is an $n$-cycle in $Y$ and $|g_n(z)|\leq D_n|z|$. Thus there is a filling $a\in C_{n+1}(Y;R)$ for $g_n(z)$ such that $|a|\leq f_n^{R, Y}(D_n|z|)+1$. 
        Then $h_{n+1}(a) + s_n(z)$ is a filling for $z$:
        \begin{align*}
            \partial(h_{n+1}(a)) + \partial(s_n(z)) &= h_{n}(\partial(a)) + \partial(s_n(z))\\
            &= h_n(g_n(z)) + (z - h_ng_n(z) - s_{n-1}(\partial(z))) = z
        \end{align*}
    and
    \begin{align*}
    |h_{n+1}(a) + s_n(z)| &\leq D'_{n+1}|a| + E_{n}|z| \\& \leq D'_{n+1}\big(f_n^{R, Y}(|g_n(z)|)+1\big) + E_{n}|z|\\& 
    \leq D'_{n+1}\big(f_n^{R, Y}(D_n|z|)+1\big) + E_{n}|z|,
    \end{align*}
    so $f_n^{R,X}\preccurlyeq f_n^{R,Y}$. The argument is symmetric, completing the proof. 
    \end{proof}

    \begin{definition}[Type $\FH_n(R)$] \label{FH}
        We say a group $G$ is of type $\FH_n(R)$ if it acts freely, faithfully, properly, cellularly, and cocompactly on a cell complex $X$ satisfying $H_i(X;R)=0$ for every $i< n$.
    \end{definition}

Note that if $X$ is a free $G$-complex with finitely many $n$-cells and $(n-1)$-cells, then the reduced simplicial chain complex $C_i(X; R)$ forms an $n$-admissible resolution and so the $n$-isoperimetric function of $X$ is equivalent to the $n$-isoperimetric function of $G$ by definition.

\begin{corollary} \label{action preserves isop}
    Suppose $G$ is a group of type $\FH_{n+1}(R)$ acting properly cocompactly on a simplicial complex $X$ such that for all $i\leq n$ we have $H_i(X;R)=0$ and $f_i^{R,G}(l),f_i^{R,X}(l)<\infty $ for all $l$. Then $f_n^{R,X}\approx f_n^{R,G}$.
\end{corollary}

\begin{corollary} \label{qiinvariant for groups}
    If $G$ and $H$ are quasi-isometric groups of type $\FH_{n+1}(R)$ and $f_i^{R,G}(l),f_i^{R,H}(l)<\infty $ for all $l$ and $i\leq n$,
    then $f_n^{R,G}\approx f_n^{R,H}$. 
\end{corollary}

Finally, we wish to show that $\hammfill{n}{R}{G}$ is a quasi-isometry invariant for groups of type $FH_n(R)$. To do this we just require the following proposition. 

\begin{proposition}\label{dfv is finite}
    Let $G$ be a group of type $FH_n(R)$. Then $\hammfill{n}{R}{G}$ takes finite values. 
\end{proposition}

\begin{proof}
    Let $X$ be a simplicial complex as in \cref{FH}. We will show that $\hammfill{n}{R}{X}(l)$ is finite. 

    We will first show that we can bound the filling norm of $n$-cycles with connected support by a function of the size of their support, then use the fact that every $n$-cycle decomposes as a sum of $n$-cycles with connected support to deduce the finiteness of the filling norm.

    Since $X$ has a free, cocompact $G$-action, there are finitely many $G$-orbits of connected subcomplexes of $X$ with $\leq l$ cells of dimension $n$. Let $K_1, \dots, K_m$ be a choice of connected subcomplexes of $X$ with $\leq l$ cells of dimension $n$ for each of these $G$-orbits.

    Given a subcomplex $Y$ of $X$, we define $N(Y)$ to be the subcomplex of $X$ consisting of $Y$ and all cells which have a boundary face in $Y$. If $Y$ is finite, we have that $N(Y)$ is finite.
    Given $i$, we use $K_i$ to filter $X$ as follows. Define $K_i^j$ inductively by $K_i^0 = K_i$ and $K_i^{j+1} = N(K_i^j)$. The finite subcomplexes $\set{K_i^j}$ form a filtration of $X$. We have that the inclusion $K_i\to X$ gives a trivial map $H_n(K_i; R)\to H_n(X; R)$. Since homology commutes with direct limits and $H_n(K_i; R)$ is a finitely generated $R$-module, we have that there is an $a_i$ such that $K_i\to K_i^{a_i}$ gives the trivial map $H_n(K_i; R)\to H_n(K_i^{a_i}; R)$. Define $L_i$ to be $K_i^{a_i}$ as above. 
    
    Let $z$ be an $n$-cycle with connected support and $|z|\leq l$. By definition, there exists $1\leq i\leq m$ and $g\in G$ such that the support of $z$ is $gK_i$. As $g^{-1}z$ and $z$ have the same filling norm, we may assume $g=1$. In particular, the support of $z$ lies in $L_i$. As $H_n(L_i;R)$ is trivial, we have that $z$ has a filling in $L_i$, which is also a filling in $X$, of norm bounded by $|L_i|$, the number of $n+1$-cells in $L_i$. We showed that for any $n$-cycle $z$ with connected support satisfying $|z|\leq l$, we have that $z$ has a filling with norm bounded by $\theta(l)=\max \{|L_i|\}$. 
    
    Now, for the general case, let $z\in C_n(X)$ be an $n$-cycle with $|z|\leq l$. Suppose that the support of $z$ is disconnected with connected components $Z_1, 
    \dots, Z_k$. Then $z$ decomposes as a sum $z_1+
    \dots +z_k$ where the support of $z_i$ is $Z_i$. Moreover, $|z| = \sum_i |z_i|$. By the definition of $\theta$, each $z_i$ has a filling $v_i$ that satisfies $|v_i|\leq \theta(|z_i|)$. Thus $\sum v_i$ is a filling of $z$ and we get $\hammfill{n}{R}{X}(l)\leq \max\set{\sum_{i=1}^m \theta(l_i)\mid \sum_{i=1}^m l_i=l,l_i\in\mathbb{N}}$, which is finite.  

    The remark above shows that the same will hold for $\hammfill{n}{R}{G}$
\end{proof}

It was necessary to assume that $G$ is of type $\FH_n(R)$ for the above proofs. However, the isoperimetric function is defined as long as $G$ is of type $\FP_n(R)$. It is known that $\FP_2(R)$ is equivalent to $\FH_2(R)$, however for $n>2$ the equivalence of $\FP_n(R)$ and $\FH_n(R)$ is unknown. This leads to the following conjecture. 

\begin{conj}
    If $G$ and $H$ are quasi-isometric groups of type $FP_{n+1}(R)$, then $f_n^G\approx f_n^H$. 
\end{conj}

Another issue with assuming $\FH_n(R)$ is that it is not known that if $G$ and $H$ are quasi-isometric, then $G$ is of type $\FH_n(R)$ if and only if $H$ is of type $\FH_n(R)$. Thus we conclude with the following question. 
\begin{question}
    Suppose $G$ and $H$ are quasi-isometric and $G$ is of type $\FH_n(R)$. Is $H$ of type $\FH_n(R)$?
\end{question}
It would suffice to show that $\FP_n(R)$ is equivalent to $\FH_n(R)$. 

\section{Subgroup retractions over arbitrary rings} \label{sec: subgroups}

In this section, we generalise results concerning subgroups and retractions of group ring modules with bounded norm, to being over more general rings. Throughout, given a resolution $F_*\twoheadrightarrow R$, $Z_n(F)$ will denote the set of $n$-cycles, that is, $Z_n(F) = \ker(F_n\to F_{n-1})$.

\begin{proposition}\cite[Proposition 4.1]{MA20}\label{prop:free cycles}
Let $G$ be a group of type $\FP_n(R)$ with $\operatorname{cd}_R G=n\geqslant 1$. Let $H<G$ be a subgroup of type $\FP_n(R)$, and $P_*\twoheadrightarrow R$ a free resolution by finitely generated $RH$-modules. Then, there exists a partial free $RG$-resolution $F_*\twoheadrightarrow R$, fitting into a split short exact sequence of $RH$-chain complexes 
\[0\to P_*\to F_*\to Q_*\to0\]
and $Z_{i}(Q)$ is $RH$-projective for all $i$. 

Moreover, there is a retraction $Z_{n-1}(F)\to Z_{n-1}(P)$.
\end{proposition}

\begin{theorem} \label{iso bounds for subgroups}
   Let $R$ be a normed ring and $G$ be a group of type $\FP_n(R)$ with $\operatorname{cd}_R =n$. If $H<G$ is a subgroup of type $\FP_n(R)$, then the $R$-homological $(n-1)$-isoperimetric class of $H$ is less than or equal to the $R$-homological $(n-1)$-isoperimetric class of $G$. 
\end{theorem}

\begin{proof}
By \cref{deffpn}, we can assume that we have a partial resolution by finitely generated free $RG$-modules. 
\begin{equation}
\begin{tikzcd}
F'_n \arrow[r] & F'_{n-1} \arrow[r] & \cdots \arrow[r] & F'_0 \arrow[r] & R \arrow[r] & 0,
\end{tikzcd}
\end{equation}
and for $H$ also:
\begin{equation}
\begin{tikzcd}
P_n \arrow[r] & P_{n-1} \arrow[r] & \cdots \arrow[r] & P_0 \arrow[r] & R \arrow[r] & 0.
\end{tikzcd}
\end{equation}

We can now apply \cref{prop:free cycles} to obtain a partial free $RG$-resolution of finite type $F_*\to R\to 0$ together with a retraction $F_i\to P_i$. Denote by $\iota_{n-1}$ the corresponding map from $Z_{n-1}(P)$ to $Z_{n-1}(F)$.

We use the resolutions $F_*$ and $P_*$ to compute the $(n-1)$-isoperimetric class of $G$ and $H$, respectively, as they are both $(n-1)$-admissible. 

Denote the boundary maps $P_n\to P_{n-1}$ and $F_n\to F_{n-1}$ by $\partial_n^H$ and $\partial_n^G$, respectively. Since both complexes are exact at $(n-1)$, we have $\image(\partial_n^H)=Z_{n-1}(P)$ and $\image(\partial_n^G)=Z_{n-1}(F)$.

Suppose that $l\in\mathbb{N}$ and $\gamma\in\image(\partial_n^H)$ such that $|\gamma|\leqslant l$. From the retraction in \cref{prop:free cycles}, $Z_{n-1}(P)$ is an internal direct summand of $Z_{n-1}(F)$, and \cref{lem:sim norm} applies with some constant $C_1$ to give
 $$
 ||\gamma||\leqslant C_1||\iota_{n-1}(\gamma)||\leqslant C_1 f^G_{n-1}(|\iota_{n-1}(\gamma)|).
 $$
Finally, by \cref{lem:morphismsareboundedwrtnorms} there is a constant $C_2$ such that $|\iota_{n-1}(\gamma)|\leq C_2|\gamma|$. Thus we obtain
 $$
 ||\gamma||\leqslant C_1 f_{n-1}^G(C_2 |\gamma|).
 $$
 This gives us $f^H_{n-1}(l)\leqslant C_1 f^G_{n-1}(C_2 l)$, i.e.
 \[
 f^H_{n-1}\preccurlyeq f^G_{n-1}.
 \qedhere\]
\end{proof}

\section[Hyperbolicity is equivalent to linear discrete filling]{Hyperbolicity is equivalent to $\hammfill{1}{R}{G}$ being linear} \label{sec: hyp implies lin}

Kielak and Kropholler proved in \cite{KK20} that if $\hammfill{1}{R}{G}$ is linear, then $G$ is hyperbolic. In this section, we prove the converse.

We will first show this on the level of spaces.
Let $(Y,d)$ be a $\delta$-hyperbolic space such that there exists an $\epsilon$-net $X$ which is uniformly locally finite, i.e. for every $d>0$ the balls $\set{B(x,d)}_{x\in X}$ have uniformly bounded size. 

Let $P_d(X)$ be the Rips complex, i.e. a flag simplicial complex whose vertex set is $X$ and there is an edge between $u,v\in X$ if $d(u,v)\leq d$.

\begin{theorem} \label{hyperbolicity implies linear for spaces}
    Let $Y$ be a hyperbolic space and $X$ a uniformly locally finite $\epsilon$-net. Let $R$ be a ring with the discrete norm. 
    Let $\partial\colon C_2(P_d(X),R)\rightarrow Z_1(P_d(X),R)$ be the boundary map.
    Then for sufficiently large $d$, we have that $\partial$ is surjective and there exists $N>0$ such that every $z\in Z_1(P_d(X),R)$ has a filling $c$ with $\abs{c}\leq N\abs{z}$.
\end{theorem}

 \begin{proof}
Fix $d>4\delta+2\epsilon$. By assumption on $X$, we have that $P_d(X)$ is finite dimensional and uniformly locally finite. Let $k$ be the maximal degree of a 0-cell in $X$ and let $N=\max\set{k+1,(k-1)(k+1)}+1$.

For $u,v,w\in X$ neighbours in $P_d(X)$ we denote by $(u,v)$ the oriented 1-cell with $\partial(u,v)=v-u$ and by $(u,v,w)$ the oriented 2-cell with $\partial(u,v,w)=(u,v)+(v,w)+(w,u)$.
     
Let $z\in Z_1(P_d(X),R)$; we will proceed by induction on $\abs{z}$.
For $\abs{z}=0$, there is nothing to prove. 

Now suppose that for all $w\in Z_1(P_d(X), R)$ with $|w|<n$ we have $b\in C_2(P_d(X), R)$ such that $\partial b = w$ and $\abs{b}\leq N\abs{w}$.

Let $z\in Z_1(P_d(X), R)$ be such that $\abs{z} = n$. 

Fix a basepoint $x_0\in X$. Let $v\in X$ be furthest away from $x_0$ in $\supp(z)$.

Let $u_1,\dots,u_l$ be the neighbours of $v$ in $\supp(z)$. There are several cases to study: 
\begin{enumerate}
    \item $l =2$ and $d(u_1, u_2)\leq d$. 
    \item $l>2$ and there exists $i, j$ such that $d(u_i, u_j)\leq d$. 
    \item For all $i, j$ we have that $d(u_i, u_j)>d$. 
\end{enumerate}

The proof will proceed as follows. 
If we are in the third case, we will find a 1-cycle $z'$ with $|z'|\leq |z|$, differing from $z$ by at most $k$ 2-cells, where $z'$ still contains $v$ in its support furthest away from $x_0$ and its neighbours satisfy case 1 or case 2.
In case 2, we will find a 1-cycle $z'$ with $|z'|\leq|z|$, differing from $z$ by the boundary of a single 2-cell such that $v$ has $(l-1)$ neighbours in $z'$. 
In case 1, we find a 1-cycle $z'$ with $|z'|<|z|$ and $z - z'$ is the boundary of a single 2-cell.
Since the degree of each vertex is bounded, we will arrive in the first case after a uniformly bounded number of steps, and then we can reduce to a smaller 1-cycle. 

\begin{figure}
\centering
\begin{minipage}{.49\textwidth}
\centering
\begin{tikzpicture}
\draw (0,0) circle [x radius=3cm, y radius=2.5cm];
\draw (0.8*-2.5cm,0.8*-1.4cm) node {\textbullet};
\draw (0.8*-2.5cm,0.8*-1.4cm) node[left] {$x_0$};
\draw (0.8*2.47cm,0.8*1.4cm) node {\textbullet};
\draw (0.8*2.45cm,0.8*1.4cm) node[above right] {$v$};
\draw (0.8*1.14cm,0.8*2.3cm) node {\textbullet};
\draw (0.8*1.14cm,0.8*2.3cm) node[above] {$u_1$};
\draw (0.8*2.49cm,0.8*-1.4cm) node {\textbullet};
\draw (0.8*2.49cm,0.8*-1.4cm) node[above right] {$u_2$};
\draw[dashed] (0.8*1.14cm,0.8*2.3cm) -- (0.8*2.49cm,0.8*-1.4cm);
\draw (0.8*1.82cm,0.8*0.45cm) node[right] {$\leq d$};
\draw (0,-2cm) node {$Y$};

\end{tikzpicture}
\caption{Case 1, $d(u_1,u_2)\leq d$.}
\label{fig:case1}
\end{minipage}
\begin{minipage}{.49\textwidth}
\centering
\begin{tikzpicture}
\draw (0,0) circle [x radius=3cm, y radius=2.5cm];
\draw (0.8*-2.5cm,0.8*-1.4cm) node {\textbullet};
\draw (0.8*-2.5cm,0.8*-1.4cm) node[left] {$x_0$};
\draw (0.8*2.47cm,0.8*1.4cm) node {\textbullet};
\draw (0.8*2.45cm,0.8*1.4cm) node[right] {$u_1$};
\draw (0.8*1.14cm,0.8*2.3cm) node {\textbullet};%u1
\draw (0.8*2.49cm,0.8*-1.4cm) node {\textbullet};
\draw (0.8*2.49cm,0.8*-1.4cm) node[right] {$u_2$};
\draw (0.8*3cm,0) node {\textbullet};
\draw (0.8*3cm,0) node[right] {$v$};
\draw (0.8*1.6cm,0.5cm) node {\textbullet};
\draw (0.8*1.6cm,0.5cm) node[below left] {$u'_1$};
\draw[dotted] (0.8*1.6cm,0.5cm) -- (0.8*1.14cm,0.8*2.3cm);
\draw[dotted] (0.8*1.6cm,0.5cm) -- (0.8*2.45cm,0.8*1.4cm);
\draw[dotted] (0.8*1.6cm,0.5cm) -- (0.8*3cm,0);
\draw[dotted] (0.8*1.6cm,0.5cm) -- (0.8*2.49cm,0.8*-1.4cm);
\draw (1.1cm,1.25cm) node[below right] {\tiny$\leq d$};
\draw (1.5cm,0.2cm) node[below right] {\tiny$\leq d$};
\draw (1.55cm,0.8cm) node[below right] {\tiny$\leq d$};
\draw (0,-2cm) node {$Y$};

\end{tikzpicture}
\caption{\centering Case 3, $d(u_1,u_2)>d$.}
\label{fig:case3}
\end{minipage}
\end{figure}

{\bf Case 1:} $l =2$ and $d(u_1, u_2)\leq d$ (the situation in $Y$ is depicted in \cref{fig:case1}). 

Let $r_i$ be the coefficient of the edge $(u_i, v)$. Since $z$ is a 1-cycle and $v$ has exactly two neighbours, we must have that $r_2 + r_1 = 0$. 

In $P_d(X)$ we have 2-cell $(u_1, v, u_2)$. Let $z' = z - r_1\partial((u_1, v, u_2))$. Hence $z' = z - r_1(u_1, v) -r_1(u_2, v) + r_1(u_1, u_2)$. From the above, we see that the coefficients of $(u_1, v)$ and $(u_2, v)$ in $z'$ are both $0$, while the coefficient of $(u_1, u_2)$ may have become non-zero. Thus we have that $|z'|\leq |z|-1$.

{\bf Case 2:} There exist $i, j$ such that $d(u_i, u_j)\leq d$. 

Without loss of generality, we can assume that $i = 1$ and $j = 2$. Let $r_i$ be the coefficient of $(u_i,v)$ in $z$.
There is an edge between $u_1$ and $u_2$ in $P_d(X)$ and one can consider the modified cycle \[z'=z-r_1(u_1,v)-r_1(u_2,v)+r_1(u_2,u_1).\] 
Note that $|z'|\leq |z|$, we have that $\partial(r_1(u_1,v,u_2))=z-z'$ and the number of neighbours of $v$ has reduced, which is what we wanted.

{\bf Case 3:} For all $i, j$ we have that $d(u_i, u_j)>d$. 

The following claim records some consequences of the $\delta$-slim definition.

\begin{claim} \label{between neighbours}
    If $u,w$ are neighbours of $v$ in $P_d(X)$, $d(u,x_0),d(w,x_0)\leq d(v,x_0)$ and $d(u,w)>d$ then 
    \begin{enumerate}
        \item $d(u,w)\leq d+2\delta$, 
        \item $\max\set{d(u,v),d(w,v)}> d-2\delta$ and
        \item $d(w,x_0)\geq d(v,x_0)-2\delta$ or $d(u,x_0)\geq d(v,x_0)-2\delta$.
    \end{enumerate} 

    Furthermore, if $d(u,x_0)\geq d(v,x_0)-2\delta$, there exists $u'$ a 0-cell of $P_d(X)$ such that for every $x\in \supp(z)$ if $d(x,u)\leq d+2\delta$ then $d(x,u')\leq d$. In particular, every neighbour of $u$ in $\supp(z)$ is a neighbour of $u'$ and $w$ is also a neighbour of $u'$.
\end{claim}

The proof of this claim is hidden in the proofs of \cite[Lemma 1.7.A]{Gromov1987} and \cite[Proposition III.$\Gamma$.1.33]{BridHäf}.

\begin{claimproof}
    By the definition of the Gromov product and $\delta$-hyperbolicity, we have the first equality and second inequality in the following equation.
\begin{align*} 
    d&<d(u,w)=-2(u,w)_{x_0}+d(u,x_0)+d(w,x_0) \\
    &\leq-2\min\set{(u,v)_{x_0},(w,v)_{x_0}}+2\delta+d(u,x_0)+d(w,x_0) \\
    &=\max\set{-d(v,x_0)+d(u,v)+d(w,x_0),-d(v,x_0)+d(w,v)+d(u,x_0)}+2\delta \\
    &\leq \max\set{d(u,v),d(w,v)}+2\delta\leq d+2\delta
\end{align*}
The second to last inequality follows from the choice of $v$ to be the furthest from $x_0$. Points 1-3 follow.

Assume $d(u,x_0)\geq d(v,x_0)-2\delta$. Consider $y\in Y$, a point on a geodesic between $x_0$ and $v$ of distance $\frac{d}{2}$ from $v$. Via a similar calculation as above we have that for every $x\in \supp(z)$
\begin{equation*} 
\begin{split}
    d(y,x)&=-2(y,x)_{x_0}+d(y,x_0)+d(x,x_0) \\&
    \leq-2\min\set{(x,u)_{x_0},(y,u)_{x_0}}+2\delta+d(y,x_0)+d(x,x_0) \\&
    =\max\set{d(x,u)+d(y,x_0)-d(u,x_0),d(y,u)+d(x,x_0)-d(u,x_0)}+2\delta \\&
    = \max\set{-\frac{d}{2}+d(x,u),-d(u,x_0)+\frac{d}{2}+d(x,x_0)}+2\delta
    \end{split}
\end{equation*}
and as $d(u,x_0)\geq d(v,x_0)-2\delta$ and $d(x,x_0)\leq d(v,x_0)$, we obtain that $d(x,x_0)\leq d(u,x_0)+2\delta $ and hence 
\[d(y,x)\leq \max\set{-\frac{d}{2}+d(x,u),\frac{d}{2}+2\delta}+2\delta.\]

Let $u'\in X$ be such that $d(u',y)\leq \epsilon$. We get that for every $x\in \supp(z)$
\[d(u',x)\leq d(u',y)+d(y,x)\leq \epsilon+\max\set{-\frac{d}{2}+d(x,u),\frac{d}{2}+2\delta}+2\delta.\qedhere\]
\end{claimproof}

By \cref{between neighbours}, $d(u_i,x_0)\geq d(v,x_0)-2\delta$ or $d(u_j,x_0)\geq d(v,x_0)-2\delta$ for every $i\neq j$. Without loss of generality, $d(u_1,v)>d(v,x_0)-2\delta$. Let $u_1'$ be as in the claim, i.e. every neighbour of $u_1$ in $\supp(z)$ is a neighbour of $u_1'$ and $u_i$ is a neighbour of $u_1'$ for every $i>1$ (corresponding situation in $Y$ depicted in \cref{fig:case3}). Note also that $u_1'$ is not in the support of $z$ as otherwise we would have neighbours of $v$ satisfying case 1 or 2 above. 

For every $x\in \supp(z)$ neighbour of $u_1$, consider the 2-cell $\sigma_x$ with vertices $(u_1,x, u_1')$, let $r_x$ be the coefficient of the 1-cell $(u_1,x)$ and consider $z'=z-\sum \partial(r_x\sigma_x)$. We have that $z'$ is a cycle as a sum of cycles. We claim that $\abs{z'}\leq \abs{z}$. 

First note that since $z$ is a 1-cycle we have that $\sum_xr_x = 0$. Since $\partial \sigma_x = (u_1,x) + (x, u_1') - (u_1, u_1')$, it suffices to consider the coefficients of the edges $(u_1,x), (x, u_1')$ and $(u_1, u_1')$, where $x$ is a neighbour of $u_1$. 

Let $r_x'$ be the coefficient of $(x, u_1')$ in $z$ which may be 0. By the construction of $z'$ we have that the coefficient of $(x, u_1')$ in $z'$ is exactly $r_x' - r_x$. The coefficient of $(u_1,x)$ is 0 and the coefficient of $(u_1, u_1')$ is $\sum_x r_x = 0$. Hence $|z'|\leq |z|$ and the equality is strict if we have $r_x' = r_x$ for some $x$. 

Finally, we see that $u_1'$ and $u_2$ are neighbours of $v$ which are adjacent in the support of $z'$, thus we can reduce to either case 1 or 2. 

Each application of case 3 requires $\leq k$ 2-cells. Each time we apply case $2$ we require one 2-cell and we apply this case at most $k-1$ times. Case 1 requires one 2-cell and hence we arrive at the desired bound for $N$. 
\end{proof}
From this, we can derive a result about groups.

\begin{corollary} \label{hyperbolicity implies linear for groups}
Let $R$ be a ring and $G$ be a hyperbolic group. Then $\hammfill{1}{R}{G}$ is linear. 
\end{corollary}

\begin{proof}
As $G$ is hyperbolic, it is of type $F_\infty$ and in particular it is of type $\FH_{2}(R)$ and is finitely generated. Let $S$ be some finite generating set and let $Y=\Cay(G,S)$ be a Cayley graph of $G$ relative to $S$. As $G$ is Gromov-hyperbolic, there exists $\delta>0$ such that $Y$ is $\delta$-hyperbolic. The group $G$ is embedded in $Y$ as the vertices, so it is a 1-net which is uniformly locally finite, as $S$ is finite. 

By \cite{Gromov1987}, for every $d>>\delta,1$ the simplicial complex $P_d(G)$ is $1$-acyclic.
By \cref{dfv is finite}, we have that $\hammfill{1}{R}{G}$ is finite-valued.
By \cref{action preserves isop}, we have that $\hammfill{1}{R}{G}\approx f_1^{R,P_d(G)}$ and by \cref{hyperbolicity implies linear for spaces} we have that $f_1^{R,P_d(G)}$ is linear. 
\end{proof}

One could attempt to use the ideas of the above proof in order to derive a similar result over other normed rings. If one assumes that $|r| = |-r|$ and that there exists an $\epsilon>0$ such that $|r|\geq \epsilon$ for all non-zero $r$, then keeping track of the norms in the above proof gives a linear isoperimetric function. However, it is not clear what happens if either of these conditions are dropped. Therefore we ask the following: 
\begin{question}
    For which normed rings $R$ do hyperbolic groups have a linear isoperimetric function over $R$? 
\end{question}

\cref{hyperbolicity implies linear for groups} shows this for any ring with the discrete norm. This is also known for $\bbZ$ with the absolute value \cite{Ger} and for $\bbR$ and $\bbQ$ with the absolute value norm \cite{PetVan}.

\begin{remark}
    The converse follows via a similar argument as in the proof of \cite[Theorem III.H.2.9]{BridHäf}, see also \cite[Proposition 4.1]{KK20}.
\end{remark}

We end this section proving the following:
\begin{theorem}\label{subgroups over arbitrary rings}
    Let $R$ be an unital ring. Let $G$ be a hyperbolic group such that $\cd_R(G) = 2$. Then a subgroup $H < G$ is hyperbolic if and only if $H$ is of type $\FP_2(R)$.
\end{theorem}
\begin{proof}
Let $G$ be a hyperbolic group. Then by \cref{hyperbolicity implies linear for groups} we have that $G$ satisfies a linear isoperimetric inequality over $R$ with the discrete norm. Thus by \cref{iso bounds for subgroups} so does $H$. By \cite[Prop 4.1]{KK20}, a linear isoperimetric inequality implies hyperbolicity, so $H$ is hyperbolic.
\end{proof}

\section[Subquadratic implies linear for discrete isoperimetric inequality]{Subquadratic implies linear for $\hammfill{1}{R}{G}$}\label{sec:subquad implies linear}

We want to show that if $G$ satisfies a subquadratic isoperimetric inequality over $R$ then it in fact satisfies a linear isoperimetric inequality and hence is hyperbolic:

\begin{theorem} \label{subquad implies lin}
    Let $R$ be a ring and assume that $G$ is $\FP_2(R)$. If $G$ satisfies a subquadratic isoperimetric inequality over $R$ then it satisfies a linear isoperimetric inequality over $R$, i.e. if $f_1^{R,G}$ is subquadratic, then it is linear.
\end{theorem}

By a result of Bowditch from 1995, if an area function satisfies two axioms (a triangle inequality and a rectangle inequality), then it is linear if it is subquadratic:

\begin{theorem} [\cite{Bowditch95}] \label{bowditch}
Let $X$ be a simplicial complex endowed with a metric such that the length of each 1-cell is 1, and let $\Omega$ be the set of all loops in the 1-skeleton of $X$.
Let us suppose that we have a map $A \colon  \Omega \to [0, \infty)$, satisfying the following two axioms:

(A1) (Triangle inequality for theta curves): If $\alpha_1,\alpha_2,\alpha_3$ are paths in the 1-skeleton of $X$ with the same endpoints, then
$A(\alpha_3^{-1}\alpha_1) \leq A(\alpha_2^{-1}\alpha_1) + A(\alpha_3^{-1}\alpha_2)$.

(A2) (Rectangle inequality): There exists $K>0$ such that if $\gamma\in\Omega$ is the concatenation of four paths $\alpha_1,\alpha_2,\alpha_3,\alpha_4$, then $A(\gamma) \geq Kd_1d_2$, where $d_1 = d(\image \alpha_1, \image \alpha_3)$ and $d_2 = d(\image \alpha_2, \image \alpha_4)$.

    Then if $$l\mapsto \sup\set{A(\gamma):\length(\gamma)=l}$$ is subquadratic, it is also linear.
\end{theorem}

Let $G$ be of type $\FP_2(R)$. 
As in the proof of \cite[Proposition 4.1]{KK20}, let $S$ be a finite generating set for $G$ and consider $\Cay(G,S)$. By the assumption that $G$ is $\FP_2(R)$, we can glue finitely many orbits of 2-cells to $\Cay(G,S)$ such that the resulting complex $X$ satisfies $H_1(X;R)=0$. We have that $X$ is a 2-dimensional $G$-CW-complex on which $G$ acts freely and cocompactly.

Let $\Omega$ be the set of all loops in $\Cay(G,S)$, which is the 1-skeleton of $X$. For a path $\gamma$ in $\Cay(G,S)$ we abusively denote by $\gamma$ the element in $C_1(X;R)$ which is supported on $\gamma$ with all coefficients being 1. Let $A\colon \Omega\to \mathbb{R}_{\geq 0}$ be $A(\gamma)=\inf\set{|c| : c, C_2(X;R), \partial c=\gamma}$. 

\begin{proposition} \label{axioms are satisfied}
    The area function $A$ satisfies axioms (A1) and (A2).
\end{proposition}

\begin{proof}
It is clear that $A$ satisfies axiom (A1).

To prove that $A$ satisfies axiom (A2), let $\gamma\in \Omega$ be the concatenation of $\alpha_1,\alpha_2,\alpha_3,\alpha_4$ and let $d_1 = d(\image \alpha_1, \image \alpha_3)$ and $d_2 = d(\image \alpha_2, \image \alpha_4)$. As $\gamma$ is a loop in the 1-skeleton, we have that $d_1,d_2\in \mathbb{N}$.

    Let $c\in C_2(X;R)$ be a filling of $\gamma$. Let $N$ be the maximal number of edges contained in the image of the attaching map of any 2-cell in X. Since the action of G on X is cocompact, the number $N$ is a well-defined integer.

    Let $D_0\subseteq \supp(c)$ be the subset of all 2-cells that intersect $\alpha_1$ in a 1-cell (perhaps several). Since every 2-cell has at most $N$ faces and since the length of $\alpha_1$ is at least $d_2$ we have that $|D_0|\geq \frac{d_2}{N}$. Set $c_1$ to be the 2-chain obtained from $c$ after setting all coefficients outside of $D_0$ to 0.

    We keep going and define $D_1,..., D_{{d_1}-1}$ recursively: in the $i$-th step we assume we defined $D_0,...,D_{i-1}$ pairwise disjoint subsets such that $D_j$ lies in the $j+1$ neighbourhood of $\alpha_1$. We denote by $c_{j+1}$ the 2-chain obtained from $c$ by setting to 0 all coefficients of cells not in $\bigcup_{l=0}^j D_l$. Note that as $i-1< d_1$ we have that $\supp(\partial(\sum_{j=1}^{i} c_j)-\alpha_1)$ contains a path from $\alpha_2$ to $\alpha_4$. 
    We set $D_i$ to be all 2-cells in $\supp(c)-\bigcup_{j=0}^{i-1} D_j$ whose boundary contains at least one 1-cell from $q_{i}=\partial (\sum_{j=1}^{i}c_j)-\alpha_1$.

    As $q_i$ contains a path from $\alpha_2$ to $\alpha_4$ we have that $|q_i|\geq d_2$, so, as explained for $|D_0|$, $|D_i|\geq \frac{d_2}{N}$, which gives us\[|\bigcup _{i=0}^{d_1-1}D_i|=\sum_{i=0}^{d_1-1}|D_i|\geq \sum_{i=0}^{d_1-1}\frac{d_2}{N}=\frac{d_1d_2}{N}.\]
    But of course $|c|\geq |\bigcup _{i=0}^{d_1-1}D_i|$. So we get \[A(\gamma)=\inf\set{|c| : c, C_2(X;R), \partial c=\gamma}\geq \frac{d_1d_2}{N}\] and setting $K=\frac{1}{N}$ we get axiom (A2) as $N$ is a constant that depends only on $X$.
\end{proof}

\cref{subquad implies lin} now follows:

\begin{proof}[Proof of \cref{subquad implies lin}]
Let $G$ be of type $\FP_2(R)$ and assume $f_1^{R,G}$ is subquadratic. In this case for $X$ and $A$ defined above we have that $l\mapsto \sup\set{A(\gamma):\length(\gamma)=l}$ is subquadratic. By \cref{axioms are satisfied}, $A$ satisfies axioms (A1) and (A2) and so by \cref{bowditch} it is linear.

    In \cite[Proposition 4.1]{KK20}, it is proven that if $G$ satisfies a linear isoperimetric inequality over $R$ then it is hyperbolic, but a slightly weaker thing assumption is used - they proved that if any loop (1-cycle which is supported on a loop and all coefficients are 1) can be filled by a 2-chain which has uniformly linearly bounded norm then $G$ is hyperbolic. So by \cite[Proposition 4.1]{KK20} we get that $G$ is hyperbolic. By \cref{hyperbolicity implies linear for groups}, we have that $f_1^{R,G}$ is linear.
\end{proof}

\section{Higher rank hyperbolicity} \label{sec: higher rank}
In this section, we will discuss higher rank hyperbolic spaces, define higher rank hyperbolic groups and derive a Gersten-like theorem for them. Higher rank hyperbolicity is defined using the language of integral currents. For an overview of the subject in the setting of proper metric spaces, see \cite{Lang11}. We give here simpler definitions that are suitable for Euclidean simplicial complexes with finitely many isometry classes of cells and are equivalent to those in \cite{Lang11} as explained in \cref{A}. Higher rank hyperbolicity is a property defined only for spaces satisfying coning inequalities. These have a general definition using integral currents, but we will give here a simpler definition for simplicial complexes instead, using the simplicial structure (see \cref{A} for a proof of the equivalence in our setting).

\begin{definition}
    A \emph{Euclidean simplicial complex} is a locally finite simplicial complex $X$, where each cell is endowed with a Euclidean metric, and $X$ is endowed with the length metric.
\end{definition}

For $X$ a Euclidean simplicial complex, we will denote by $(C_*(X),\partial)$ the simplicial chain complex. Then $Z_k(X)=\ker(\partial\colon C_k(X)\to C_{k-1}(X))$ and \linebreak$M\colon C_k(X)\to \bbR_{\geq 0}$ is defined on $k$-cells by $M(\sigma)=\operatorname{Vol}(\sigma)$ and extended linearly.

When there exists uniform lower and upper bounds on the volume of cells in $X$, $X$ is bi-Lipschitz equivalent to a simplicial complex for which there exist $v_k$ for every $k\leq n$ such that if $\sigma$ is a $k$-cell, we have $\operatorname{Vol}(\sigma)=v_k$. Thus, $M$ is linearly or bi-Lipschitz equivalent to the $\ell^1$-norm on $C_k(X;\bbZ)$.

\begin{definition}[Coning inequalities]
    A simplicial complex $X$ satisfies $(CI_n)$ if there exists a constant $c$ such that any two points $x,x'\in X$ can be joined by a path of length at most $cd(x,x')$, and for $x_0\in X, r>0, k \in\set{ 1, . . . , n}, R\in Z_k(X)$, such that $\supp(R)\subset B_r(x_0)$, there exists $S\in C_{k+1}(X)$ such that $\partial S=R$ and $M(S) \leq c r M(R)$.
\end{definition}

\begin{definition}
    A Euclidean simplicial complex $X$ that satisfies $CI_{n-1}$ is \emph{rank-$n$ hyperbolic} if the simplicial chain complex satisfies a sub-Euclidean isoperimetric inequality in dimension $n$. That is, for every $n$-cycle $Z$, there exists a filling $S$ such that $M(S)\leq CM(Z)^\frac{n+1}{n}$.

    We say that a group $G$ is \emph{rank-$n$ hyperbolic} if there exists a simplicial metric $K(G,1)$ with a finite $(n+1)$-skeleton such that its universal cover is rank-$n$ hyperbolic.
\end{definition}

See \cite{GoLa} for equivalent definitions of higher rank hyperbolic spaces.

\begin{corollary}\label{higherranksubgroup}
    Let $G$ be a rank-${n}$ hyperbolic group with $\operatorname{cd} G=n+1$. If $H<G$ is a subgroup for which there exists a simplicial metric $K(H,1)$ with a finite $(n+1)$-skeleton such that its universal cover satisfies $(CI_n)$, then $H$ is rank-$n$ hyperbolic.
\end{corollary}

\begin{proof}
    By definition, there exists $X$, the universal cover of a $K(G,1)$ with a finite $(n+1)$-skeleton that satisfies a sub-Euclidean isoperimetric inequality. So the chain complex $C_*(X)$ satisfies a sub-Euclidean isoperimetric inequality in dimension $n$. By \cref{iso class is independent}, every $(n+1)$-admissible $\bbZ G$ resolution of $\bbZ$ satisfies a sub-Euclidean isoperimetric inequality in dimension $n$.

    Let $Y$ be a simplicial metric $K(H,1)$ with a finite $(n+1)$-skeleton such that its universal cover satisfies $(CI_n)$. From \cref{iso bounds for subgroups}, we get that any $n$-admissible $\bbZ H$ resolution of $\bbZ$ satisfies a sub-Euclidean isoperimetric inequality in dimension $n$. In particular, the resolution $C_*(Y)$ does.
\end{proof}
\begin{corollary}
    Let $G$ be a rank-${n}$ hyperbolic group with $\operatorname{cd}G=n+1$. Then if $H<G$ is a subgroup of type $F_{n+1}$ of finite Nagata dimension, then $H$ is rank-$n$ hyperbolic.
\end{corollary}

\begin{proof}
    Similarly to the previous corollary, $H$ satisfies a sub-Euclidean isoperimetric inequality. By \cite[Corollary 1.3]{BasWenYou}, we have that $H$ also satisfies $(CI_n)$.
\end{proof}

In general, one cannot hope that subgroups with sufficiently strong finiteness properties satisfy $(CI_n)$. In \cite[Section 2.5.2]{BRSFinGen} a right-angled Artin group is constructed with a subgroup that is of type $F$ and has a quartic Dehn function. In contrast, groups satisfying $(CI_n)$ necessarily satisfy a quadratic Dehn function \cite{WengerIsoConing}.

We showed that for a group $G$ and a general ring $R$, $\hammfill{1}{R}{G}$ is linear if and only if $FV_1^{\mathbb{Z},G}$ is linear. This is because both are equivalent to the hyperbolicity of $G$. This leads us to ask whether the same holds for higher rank hyperbolicity, namely:
\begin{question}
    Is $\hammfill{n}{R}{G}$ being sub-Euclidean for some $R$ equivalent to $FV_n^{\mathbb{Z},G}$ being sub-Euclidean?
\end{question}

\appendix
\section[Higher rank hyperbolicity: simplicial definitions]{Higher rank hyperbolicity, simplicial definitions \newline\small{By Shaked Bader}} \label{A}

In \cite{KL20} it is noted that for a metric simplicial complex one can define coning inequalities using integral currents or using simplicial chains and that those definitions are equivalent due to a variant of the Federer–Fleming deformation theorem \cite{FedFlem}, but no proof is given. We will give a proof of this fact (under mild assumptions on the complex) using a suitable variant given in the appendix of \cite{BasWenYou}.
We will prove similarly and under the same assumptions that the isoperimetric functions of a Euclidean simplicial complex $X$ with respect to the integral currents chain complex and the simplicial chain complex are equivalent. 
This shows that the equivalent definitions for higher rank hyperbolicity given in \cite{GoLa} are equivalent to the definitions given in \cref{sec: higher rank}.

\subsection{Metric currents}
We refer the reader to \cite{Lang11} for more background regarding metric currents and integral currents. The definition of integral currents given here follows from a theorem in Section 8 of \cite{Lang11}.

Let $X$ be a locally compact metric space. We consider the space of compactly supported Lipschitz functions from $X$ to $\mathbb{R}$, denoted $\mathscr{D}(X)$ and the space of locally Lipschitz maps from $X$ to $\bbR$, denoted $\Liploc(X)$.

\begin{definition}
    An $n$-dimensional metric current on $X$ is a functional $T\colon\mathscr{D}(X)\times [\Liploc(X,\bbR)]^n\to \bbR$ such that:

    \begin{enumerate}
        \item (Multilinearity) $T$ is $n+1$-linear.
        \item (Continuity) $T(f^k,\pi_1^k,..., \pi_n^k) \to T(f,\pi_1,..., \pi_n)$ whenever $\pi_i^k\to_{k\to \infty} \pi_i,f^k\to f$ pointwise on $X$, $\sup_k\Lip(\pi_i^k|_K) < \infty$ for every compact set $K \subseteq X$ and $i\in\set{1,...,n}$ and $\bigcup_k \spt(f,k)$ sits inside some compact subset of $X$.
        \item (Locality) $T(f,\pi_1,...,\pi_n)=0$ if there exists $i$ such that $\pi_i|_{\supp(f)}\equiv c$ for some constant $c$.
    \end{enumerate}
\end{definition}

\begin{definition}[support]
Given an $n$-dimensional current $T$, its \emph{support}, $\supp(T)$, is the intersection of all closed sets $C\subseteq X$ with the property that $T(f,\pi_1,...,\pi_n)=0$ whenever $\supp(f)\cap C=\emptyset$.
\end{definition}

Our motivating example is the following:
\begin{example} \label{standard example}
    Given $u\in L^1(\bbR^n)$ and a bi-Lipschitz map $F\colon\supp(u)\to X$, we can define the $n$-current $F_\sharp[u]$ to be \[F_\sharp[u](f,\pi_1,...,\pi_n)=\intop_{\bbR^n} f\circ F(x_1,...,x_n)\cdot \det(d(\pi_i\circ F))\] when $f\circ F\in L^1(\bbR^n)$ and use the fact that $L^1(\bbR^n)$ is dense in the compactly supported Lipschitz functions, to define $T$ on $\mathscr{D}(X)\times [\Liploc(X,\bbR)]^n$.
\end{example}

We can now define the mass of a current which is a generalization of volume, i.e. in the case of \cref{standard example} when $u=1_W$ we should get the volume of $F(W)$.

\begin{definition}
    Let $T$ be an $n$-current. We define the \emph{mass of $T$}, $M(T)$, to be the supremum of $\sum_i T(f_i,\pi^i_1,...,\pi^i_n)$ where the supremum goes over all finite families of $\set{(f_i,\pi^i_1,...,\pi^i_n)}$ such that $\sum_i |f_i|\leq 1$ and $\pi^i_j|_{\supp(f_i)}$ are 1-Lipschitz.
\end{definition}

We have a boundary map $\partial$ from $n$-currents to $(n-1)$-currents defined by $\partial(T)(f,\pi_1,...,\pi_{n-1})=T(\chi,f,\pi_1,...,\pi_{n-1})$ where $\chi$ is a compactly supported Lipschitz function which is 1 on the support of $f$ (this is well defined by locality). By locality again, we have that $\partial\circ \partial =0$.

We will restrict ourselves to currents with finite mass, and as we want to use a chain complex structure we will restrict ourselves to normal currents:

\begin{definition}
    We define the \emph{compactly supported normal $m$-currents} to be the $m$-currents $T$ such that $M(T)+M(\partial T)<\infty$ and there exists a compact set $K$ such that the value of $T(f,\pi_1,...,\pi_n)$ depends only on the values of $f|_K,\pi_1|_K,...,\pi_n|_K$, i.e. if $(f|_K,\pi_1|_K,...,\pi_m|_K)=(f'|_K,\pi'_1|_K,...,\pi'_m|_K)$ then $T(f,\pi_1,...,\pi_m)=T(f',\pi'_1,...,\pi'_m)$.

    We denote the chain complex of compactly supported normal currents by $N_*(X)$.
\end{definition}

We now define integral currents to be currents you get by ``glueing'' together currents of the form given in \cref{standard example}.

\begin{definition}
    We say a normal compactly supported $n$-current $T$ is \emph{integral} if there exists $u_i\in L^1(\bbR^n,\bbZ)$ and bi-Lipschitz maps $F_i\colon\supp(u_i)\to X$ such that $T=\sum_{i=1}^\infty F^i_\sharp[u_i]$.

    We denote the integral $n$-currents on $X$ by $I_n(X)$.
\end{definition}

It is easily checked that $(I_*(X),\partial)$ is a chain complex, so one can consider the isoperimetric functions with respect to $\partial$ and the mass, that is $$F^{\text{Curr}}_n(l)=\sup_{b\in \image(\partial\colon I_{n+1}(X)\to I_n(X)),M(b)\leq l}\inf\{M(c)\mid c\in I_{n+1}(X),\partial(c)=b\}.$$

We consider the simplicial chain complex $(C_*(X),\partial)$ as sitting inside $(I_*(X), \partial)$ via the embedding of $C_*(X)$ in the Lipschitz singular chain complex and the chain complex map $[\cdot]\colon C_*^\text{Lip}(X)\to I_*(X)$ which is defined to be additive and for $A$ a Lipschitz map from the $n$-simplex $\Delta^n$ to $X$ we define $[A]=A_\sharp[1_{\Delta^n}]$.

\subsection{Isoperimetric and coning inequalities}
Throughout, we assume that $X$ is a locally finite simplicial complex, each cell is endowed with a Euclidean metric and there are finitely many isometry types of cells. In particular, $X$ is finite dimensional. 

We remind the reader that one can consider the simplicial chain complex with $\mathbb{Z}$ coefficients endowed with the $\ell^1$-norm and the isoperimetric function $f_n^X$ derived from it, as in \cref{def isop for spaces}.

\begin{theorem} \label{lin isoper ineq}
    Let $X$ be a simplicial complex of dimension $n$, each cell endowed with a Euclidean metric such that there are finitely many isometry classes of cells. Then the isoperimetric functions $f_n:=f_n^X$ and $F_n^\text{Curr}$ are equivalent.
\end{theorem}

To prove this theorem, we will use the fact that every integral current is close to a simplicial current. This is encoded in the following theorem:

\begin{theorem}[{\cite[Theorem A.2]{BasWenYou}}] \label{deformation theorem}
    Let $X$ be a simplicial complex of dimension $n$, each cell endowed with a Euclidean metric such that there are finitely many isometry classes of cells.
    There exists a constant $C>0$ such that for $T\in I_k(X)$, there exists $P$ a $k$-simplicial chain, $R\in I_k(X)$ and $S\in I_{k+1}(X)$ such that 
    $T=P+R+\partial S$, and
        \begin{align*}
                  &  M(P)\leq  CM(T), \\&
                    M(\partial P)\leq  CM(\partial T), \\&
                    M(R)\leq  CM(\partial T)\text{ and} \\&
                    M(S)\leq  CM(T). \\
        \end{align*}
    Furthermore, denoting by $\Hull(A)$ the smallest subcomplex containing $A$, we have \[\supp(R),\supp(\partial P)\subseteq \Hull(\supp(\partial T)),\] and \[\supp(P),\supp(S)\subset\Hull(\supp(T)).\]
\end{theorem}

\begin{observation}
    If $T$ is a cycle we get that $R=0$.
\end{observation}

\begin{proof} [Proof of \cref{lin isoper ineq}]
In this proof, we identify between simplicial chains and their image in the integral currents.
$F_n^\text{Curr}\preccurlyeq f_n$:

Let $T\in I_n(X)$ be a boundary. By \cref{deformation theorem} there exists $P,R,S$ as in the theorem and as $T$ is a cycle $R=0$. Let $V$ be a cellular filling of $P$ such that $M(V)\leq f_n(M(P))$, which exists by the definition of $f_n$. We have that $\partial(V+S)=\partial(V)+\partial(S)=P+\partial S=T$ and $M(V+S)\leq M(V)+M(S)\leq f_n(M(T))+C(M(T))$, so $F_n^\text{Curr}(T)\preccurlyeq f_n(T)$.

$F_n^\text{Curr}\succcurlyeq f_n$:

Let $V$ be an $n$-simplicial boundary. There exists $T\in I_{n+1}(X)$ filling $V$ such that $M(T)\leq F_n^\text{Curr}(M(V))+1$. 
By \cref{deformation theorem}, $T=P+R+\partial S$ as in the theorem. 
We have that $\supp(R)\subseteq \Hull(\supp(\partial T))\subseteq X^{(k)}$, and as $R$ is an $(n+1)$-current this gives us that $R=0$. So $\partial(P)=\partial T=V$ and $M(P)\leq CM(V)\leq C(F_n^\text{Curr}(M(T))+1)$.
\end{proof}

\begin{remark}
   One can define the isoperimetric functions corresponding to the chain complex of Lipschitz maps and, using the version of the deformation theorem in \cite[Chapter 10.3]{Epstein}, prove similarly that the isoperimetric functions corresponding to it are equivalent to the simplicial and current ones.
\end{remark}

In \cite{KK20} it is shown that a $\PD_n$ group is either amenable or satisfies a linear simplicial isoperimetric inequality in dimension $n-1$. Using that and \cref{lin isoper ineq} we get the following:
\begin{corollary}
    A non-amenable $\PD_n$ group with a finite metric $K(G,1)$ whose universal cover satisfies $(CI_n)$ is rank-$(n-1)$ hyperbolic.
\end{corollary}

Similarly to \cref{lin isoper ineq}, one can prove that coning inequalities do not depend on the chain complex.

\begin{definition}
    A simplicial complex $X$ satisfies an \emph{$n$-simplicial (respectively $n$-integral) coning inequality} if there exists a constant $c$ such that any two points $x,x'\in X$ can be joined by a curve of length $\leq cd(x,x')$, and for $x_0\in X, r>0, k \in\set{ 1, . . . , n}$ and a $k$-simplicial cycle $R$ (respectively $R\in I_k(X), \partial R=0$) such that $\supp(R)\subset B_r(x_0)$ there exists $S\in C_{k+1}(X)$ (respectively $S\in I_{k+1}(X)$) such that $\partial S=R$ and $M(S) \leq c r M(R)$.
\end{definition}

\begin{theorem}
    Let $X$ be a simplicial complex of dimension $n$, each cell endowed with a Euclidean metric such that there are finitely many isometry classes of cells and finitely many types of links. Then $X$ satisfies an $n$-simplicial coning inequality if and only if it satisfies an $n$-integral coning inequality.
\end{theorem}

The proof uses the same trick as the proof of \cref{lin isoper ineq}.

\begin{proof}
Assume that $X$ satisfies an $n$-integral coning inequality with constant $c$ and let $R$ be a $k$-simplicial cycle for some $1\leq k\leq n$ with $\supp(R)\subseteq B_r(x_0)$. Then, identifying $R$ with its image in $I_k(X)$ and using the $n$-integral coning inequality, we get a filling $T\in I_{k+1}(X)$ satisfying $M(T)\leq crM(R)$. By \cref{deformation theorem} and as explained in the proof of the $F_n^\text{Curr}\succcurlyeq f_n$ case, we have $T=P+\partial S$ with $P$ a simplicial $(k+1)$-cycle, and $\partial P=\partial T=R$, $M(P)\leq CM(T)\leq CcrM(R)$, so $X$ satisfies an $n$-simplicial coning inequality with constant $cC$.

Assume that $X$ satisfies an $n$-simplicial coning inequality with constant $c$. As there are finitely many isometry classes of cells, there exists $r_0>0$ such that for every $x\in X,r\leq r_0$ $B_r(x)$ is isometric to the cone over the link of $x$. There exists $L>0$ such that the homotopies from $B_r(x)$ to $x$ along geodesics are $L$-Lipschitz. As explained in \cite{KL20} Section 2.7, there exists $D>0$ such that for every $x\in X$ we have that $B_{r_0}(x)$ satisfies an $n$-integral coning inequality with constant $D$.

Let $T\in I_k(X), \partial T=0$ for some $1\leq k\leq n$ with $\supp(T)\subseteq B_r(x_0)$. Then, if $r<r_0$ we have that there exists $V\in I_{k+1}(X)$ such that $\partial V=T$, $M(V)\leq DrM(T)$; otherwise, by \cref{deformation theorem} and as explained in the proof of the $f_n\succcurlyeq F_n^\text{Curr}$ case, we have $T=P+\partial S$ with $P$ a simplicial $k+1$-cycle. By the $n$-simplicial coning inequality we get a filling of $P$, $V\in I_{k+1}(X)$, satisfying $M(V)\leq crM(P)$. So $\partial (V+S)=T$ and \begin{equation*}
    \begin{split}
        M(V+S) &\leq M(V)+M(S)\leq crM(P)+CM(T) \\& 
        \leq cCrM(T)+CM(T)
        =cCrM(T)+r_0\frac{C}{r_0}M(T) \\&
        \leq cCrM(T)+r\frac{C}{r_0}M(T)=(cC+\frac{C}{r_0})rM(T).
    \end{split}
\end{equation*}
so $X$ satisfies an $n$-integral coning inequality with constant $\max\{cC+\frac{C}{r_0},D\}$.
\end{proof}

\AtNextBibliography{\scriptsize}
\printbibliography
\end{document}